\documentclass[12pt]{amsart}
\usepackage[utf8]{inputenc}

\usepackage[margin=1.3in]{geometry}

\usepackage{amsmath}
\usepackage{amsfonts}
\usepackage{amssymb}
\usepackage{amsthm}
\usepackage{mathtools}
\usepackage{caption}
\usepackage{subcaption}
\usepackage{bbm}
\usepackage[export]{adjustbox}

\usepackage[all]{xy}

\usepackage{tikz-cd}
\usetikzlibrary{matrix}
\usepackage{graphicx} 
\usepackage{epstopdf}

\usepackage[linktocpage]{hyperref}
\hypersetup{
    colorlinks=true,
    linkcolor=blue,
    citecolor=blue,      
    urlcolor=blue,
}

\newtheorem{theorem}{Theorem}
\newtheorem{definition}[theorem]{Definition}
\newtheorem{proposition}[theorem]{Proposition}
\newtheorem{lemma}[theorem]{Lemma}

\newtheorem{corollary}[theorem]{Corollary}

\newtheorem{remark}[theorem]{Remark}

\numberwithin{equation}{section}
\numberwithin{theorem}{section}

\usepackage[english]{babel}
\usepackage{enumitem}

\usepackage{hyphenat}

\usepackage[backend=bibtex,style=alphabetic,maxalphanames=4,maxnames=4]{biblatex}

\renewbibmacro{in:}{}

\DeclareDelimFormat[bib,biblist]{nametitledelim}{\addcomma\space}

\DeclareFieldFormat*{title}{\mkbibitalic{#1}\addcomma}
\DeclareFieldFormat*{journaltitle}{#1}
\DeclareFieldFormat*{volume}{\mkbibbold{#1}}
\DeclareFieldFormat{pages}{#1}
\DeclareFieldFormat[misc]{date}{preprint {#1}}
\DeclareFieldFormat{mr}{%
  MR\addcolon\space
  \ifhyperref
    {\href{http://www.ams.org/mathscinet-getitem?mr=MR#1}{\nolinkurl{#1}}}
    {\nolinkurl{#1}}}
    
\AtEveryBibitem{
  \clearfield{url}
  \clearfield{number}
  \clearfield{doi}
  \clearfield{issn}
  \clearfield{isbn}
  \clearfield{eprintclass}
}
\AtEveryBibitem{\ifentrytype{book}{\clearfield{pages}}{}}

\bibliography{hecke}

\usepackage{fancyhdr}
\title{An example of higher-dimensional Heegaard Floer homology}

\author{Yin Tian}
\address{School of Mathematical Sciences, Beijing Normal University; 
Laboratory of Mathematics and Complex Systems, Ministry of Education, Beijing 100875, China}
\email{yintian@bnu.edu.cn} \urladdr{}

\author{Tianyu Yuan}
\address{Beijing International Center for Mathematical Research, Peking University, Beijing 100871, China}
\email{ytymath@pku.edu.cn} \urladdr{}

\date{\today}

\keywords{Higher-dimensional Heegaard Floer homology, Hecke algebra}

\subjclass[2010]{Primary 53D10; Secondary 53D40.}

\begin{document}

\maketitle

\begin{abstract}
We count pseudoholomorphic curves in the higher-dimensional Heegaard Floer homology of disjoint cotangent fibers of a two dimensional disk. We show that the resulting algebra is isomorphic to the Hecke algebra associated to the symmetric group. 
\end{abstract}

\section{Introduction}
Many topological properties of a manifold $M$ can be recovered from the symplectic geometry of its cotangent bundle $T^*M$. 
An example is the $A_\infty$-equivalence between the wrapped Floer homology $CW^*(T^*_qM)$ of a cotangent fiber and the space $C_{-*}(\Omega_qM)$ of chains on the based loop space of $M$, proved by Abbondandolo, Schwarz \cite{abbondandolo2010floer} and Abouzaid \cite{abouzaid2012wrapped}. 

On the symplectic side, there is a generalization, the wrapped Floer homology $CW^*(\sqcup_{i=1}^{\kappa} T^*_{q_i}M)$ of $\kappa$ disjoint cotangent fibers in the framework of {\em higher-dimensional Heegaard Floer homology} (abbreviated HDHF) established by Colin, Honda and Tian \cite{colin2020applications}.
It is related to the braid group of $M$ on the topological side.

When $M=\Sigma$ is a real oriented surface, the HDHF was recently studied by Honda, Tian and Yuan \cite{honda2022higher}.  
Pick $\kappa$ basepoints $\mathbf{q}=\{q_1,\dots,q_\kappa\}\subset \Sigma$. 
By definition, $CW^*(\sqcup_{i} T^*_{q_i}\Sigma)$ is an $A_{\infty}$ algebra over $\mathbb{Z}[[\hbar]]$, where $\hbar$ keeps track of the Euler characteristic of the holomorphic
curves that are counted in the definition of the $A_{\infty}$-operations. 
If $\Sigma$ is not a two sphere, then $CW^*(\sqcup_{i} T^*_{q_i}\Sigma)$ is supported in degree zero and hence is an ordinary algebra.
The main result of \cite{honda2022higher} is that the algebra $CW^*(\sqcup_{i} T^*_{q_i}\Sigma)$ is isomorphic to the {\em braid skein algebra} $\operatorname{BSk}_{\kappa}(\Sigma)$ of $\Sigma$, which was defined by Morton and Samuelson \cite{morton2021dahas}. 
Roughly speaking, $\operatorname{BSk}_{\kappa}(\Sigma)$ is a quotient of the group algebra of the braid group of $\Sigma$ by the {\em HOMFLY skein relation} which is expressed in terms of $\hbar$. 
The skein relation has an explanation as holomorphic curve counting due to Ekholm and Shende \cite{ekholm2021skeins}. 
This is one of the keys to build the bridge between $CW^*(\sqcup_{i} T^*_{q_i}\Sigma)$ and $\operatorname{BSk}_{\kappa}(\Sigma)$.

Morton and Samuelson showed that $\operatorname{BSk}_{\kappa}(\Sigma)$ is isomorphic to the {\em double affine Hecke algebra} associated to $\mathfrak{gl}_{\kappa}$ when $\Sigma$ is a torus. 
Based on this result, Honda, Tian and Yuan proved the isomorphisms between $CW^*(\sqcup_{i} T^*_{q_i}\Sigma)$ and various Hecke algebras of type A for $\Sigma$ being a disk, a cylinder or a torus. 

In this paper, we focus on the local case: $\Sigma=D^2$ is a disk.
Let $\operatorname{End}(L^{\otimes \kappa})$ denote the algebra $CW^*(\sqcup_{i} T^*_{q_i}D^2)$ throughout the paper.
It is isomorphic to the finite Hecke algebra associated to the symmetric group $S_{\kappa}$ over $\mathbb{Z}[[\hbar]]$ \cite{honda2022higher}.   
The main result of this paper is to show that $\operatorname{End}(L^{\otimes \kappa})$ can be defined over $\mathbb{Z}[\hbar]$, and the isomorphism to the finite Hecke algebra still holds.


By definition, the Floer generators of $\operatorname{End}(L^{\otimes \kappa})$ are tuples of intersection points between the cotangent fibers $T^*_{q_i}D^2$. 
They are in one-to-one correspondence to elements of the symmetric group $S_{\kappa}$. 
Let $T_w \in \operatorname{End}(L^{\otimes \kappa})$ denote the corresponding Floer generator for $w \in S_{\kappa}$.

We now recall some basic facts about the Hecke algebra associated to $S_{\kappa}$.  
For our purpose, we change the variable from $q$ to $\hbar$ via $\hbar=q-q^{-1}$. 

\begin{definition}
    The Hecke algebra $H_\kappa$ is a unital $\mathbb{Z}[\hbar]$-algebra generated by $\tilde{T}_1,\dots,\tilde{T}_{\kappa-1}$, with relations
    \begin{align*}
        \tilde{T}_{i}^2&=1+\hbar \tilde{T}_i,\\
        \tilde{T}_i\tilde{T}_j&=\tilde{T}_j\tilde{T}_i\,\,\,\mathrm{for}\,\,\,|i-j|>1,\\
        \tilde{T}_i\tilde{T}_{i+1}\tilde{T}_i&=\tilde{T}_{i+1}\tilde{T}_{i}\tilde{T}_{i+1}.
    \end{align*}
\end{definition}

It is known that the Hecke algebra $H_\kappa$ is a free $\mathbb{Z}[\hbar]$-module with a basis $\tilde{T}_w, w\in S_{\kappa}$, called the {\em standard basis}. 
Here, $\tilde{T}_i=\tilde{T}_{s_i}$ for the transposition $s_i=(i,i+1)$. 
There is a length function on $S_{\kappa}$ defined by $l(w)=\min\{l~|~w=s_{i_1}\cdots s_{i_l}\}$. 
The basis $\tilde{T}_w=\tilde{T}_{i_1}\cdots \tilde{T}_{i_l}$ if $w=s_{i_1}\cdots s_{i_l}$ is an expression of minimal length. 
Moreover, the algebra structure on $H_\kappa$ is uniquely determined by
     \begin{align} \label{eq-hecke}
        \tilde{T}_i\tilde{T}_w=\left\{
        \begin{array}{ll}
        \tilde{T}_{s_iw} & \mathrm{if}\,\,\,l(s_iw)>l(w)+1\\
               \tilde{T}_{s_iw}+\hbar\tilde{T}_w & \mathrm{if}\,\,\,l(s_iw)<l(w)-1
    \end{array}\right.
    \end{align}

Our main result is the following. 

\begin{theorem} \label{main-thm}
The HDHF homology $\operatorname{End}(L^{\otimes \kappa})$ is defined over $\mathbb{Z}[\hbar]$. 
Moreover, there is an isomorphism of unital algberas $\phi: H_\kappa \to \operatorname{End}(L^{\otimes \kappa})$ such that $\phi(\tilde{T}_w)=T_w$ for $w \in S_{\kappa}$.
\end{theorem}

Our strategy of the proof is a direct computation of HDHF. It is different from the method in \cite{honda2022higher}. 
For the convenience of curve counting, we view the disk $D^2$ as a product of two intervals $I_1\times I_2$ so that $T^*D^2=T^*I_1\times T^*I_2\subset \mathbb{C}\times\mathbb{C}$. 
Lagrangians (wrapped cotangent fibers) are products of 1-dimensional Lagrangians in $T^*I_1$ and $T^*I_2$.
Counting curves in $T^*D^2$ is then reduced to counting curves in $\mathbb{C}$, which is easier to work with.

\vspace{.2cm}
\noindent {\em Further directions:}

It is natural to ask whether the HDHF homology $\operatorname{End}(\sqcup_{i} T^*_{q_i}\Sigma)$ of disjoint fibers of $T^*\Sigma$ can be defined over $\mathbb{Z}[\hbar]$ for a general surface $\Sigma$. 
It is possible to generalize our local result to the global case by using some sheaf theoretic technique, for instance in \cite{ganatra2018sectorial}.  

It is also interesting to explain the geometric meaning of the change of variables $\hbar=q-q^{-1}$. Note that the {\em canonical basis} of the Hecke algebra is defined over $\mathbb{Z}[q,q^{-1}]$. We will express the canonical basis via HDHF in an upcoming paper.

\vspace{.2cm}
\noindent\textit{Acknowledgements}. The authors thank Ko Honda for numerous ideas and suggestions. 
YT is supported by NSFC 11971256.

\section{Preliminaries}
\label{section-lag}
We first specify the ambient manifold and Lagrangian submanifolds of interest. 
For convenience of curve counting, we set $D^2=I_1\times I_2$ with $I_1=I_2=[0,1]$, which is topologically the same as the unit disk. 
Let $X=T^*D^2=T^*I_1\times T^*I_2$ be the total space of the cotangent bundle of $D^2$.

Consider the canonical Liouville form $\theta$ on $T^*D^2$, which induces a contact manifold structure at the infinity of $(T^*D^2,\theta)$.
For a Lagrangian $L\subset{T^*D^2}$, denote its boundary at infinity by $\partial_\infty L$.
An isotopy of Lagrangians $L_t$ in $T^*D^2$ is called \emph{positive} if $\alpha(\partial_t\partial_{\infty} L_t)>0$ for all $t$.
Let $T^*_vD^2=T^*D^2|_{\partial D^2}$ be the vertical boundary of $T^*D^2$ over $\partial D^2$. 
We require that any isotopy $L_t$ cannot cross $T_v^*D^2$.
A positive isotopy is also called a ``partially wrapping''.
For the details of partially wrapped Fukaya categories, we refer to \cite{sylvan2019partially} by Sylvan and \cite{ganatra2020covariantly} by Ganatra, Shende, and Pardon. 

We next consider the generalization to wrapped HDHF.
Pick $\kappa$ disjoint basepoints $\mathbf{q}=\{q_1,\dots,q_\kappa\} \subset D^2\backslash \partial D^2$ and consider the $\kappa$ cotangent fibers $L_{0i}=T^*_{q_i}D$ for $i=1,\dots,\kappa$. 
We denote $\mathcal{L}_{0}=\{L_{01},\dots,L_{0\kappa}\}$.
An isotopy of $\kappa$-tuple of Lagrangians $\mathcal{L}_t=\{L_{t1},\dots,L_{t\kappa}\}$ is called {\em positive} if $\alpha(\partial_t\partial_\infty L_{ti})>0$ for all $i=1,\dots,\kappa$ and all $t$. 
For a pair of $\kappa$-tuples of Lagrangians $\mathcal{A}$ and $\mathcal{B}$, we denote $\mathcal{A}\rightsquigarrow\mathcal{B}$ if there is an positive isotopy from $\mathcal{A}$ to $\mathcal{B}$.

We then perform positive wrapping on $\mathcal{L}_0$ to get $\mathcal{L}_j=\{L_{j1},\dots,L_{j\kappa}\}$ for $j=1,2$. 
Specifically, we put $\mathcal{L}_0,\mathcal{L}_1,\mathcal{L}_2$ in the position as in Figure \ref{figure-X} (b,c), which represent the $T^*I_1$-direction and $T^*I_2$-direction, respectively. 
It is easy to check that
    $$\mathcal{L}_0\rightsquigarrow \mathcal{L}_1\rightsquigarrow \mathcal{L}_2.$$

\begin{remark}
    We fix this special wrapping of $\mathcal{L}_0,\mathcal{L}_1,\mathcal{L}_2$. It is crucial for our counting of curves. 
    We do not know whether the finite generation over $\hbar$ still holds for a general wrapping.
\end{remark}

The HDHF cochain complex $CW^*(\mathcal{L}_i,\mathcal{L}_j)$, $i<j$ is defined as the free abelian group generated by $\kappa$-tuples of intersection points between $\mathcal{L}_i$ and $\mathcal{L}_j$ over $\mathbb{Z}[[\hbar]]$. 
By definition, $CW^*(\mathcal{L}_i,\mathcal{L}_j)$ is an $A_\infty$-algebra.
We refer the reader to \cite{honda2022higher} for details of the definition of HDHF in this case.
 
There is an absolute grading on $CW^*(\mathcal{L}_i,\mathcal{L}_j)$ and the degree is supported at $0$ by \cite[Proposition 2.9]{honda2022higher}. Hence, $CW^*(\mathcal{L}_i,\mathcal{L}_j)$ is an ordinary algebra over $\mathbb{Z}[[\hbar]]$. 
We denote it by $\operatorname{End}(L^{\otimes\kappa})$.

We describe the $m_2$-composition map of $\operatorname{End}(L^{\otimes\kappa})$ in the following. 
We set $\widehat{X}=D_3\times X$ as the target manifold, where $D_3$ is the unit disk with 3 boundary punctures and referred as the ``$A_\infty$ base direction''.
Let $z_0,z_1,z_2$ be the boundary punctures of $D_3$ and let $\alpha_0,\alpha_1,\alpha_2$ be the boundary arcs.
We extend $\mathcal{L}_i$ to the $D_3$ direction by setting $\widehat{\mathcal{L}}_i=\alpha_i\times \mathcal{L}_i$, $\widehat{L}_{ij}=\alpha_i\times L_{ij}$, $i=0,1,2$, $j=1,\dots,\kappa$, which are Lagrangian submanifolds of $\widehat{X}$.

\begin{figure}
     \centering
     \includegraphics[width=13cm]{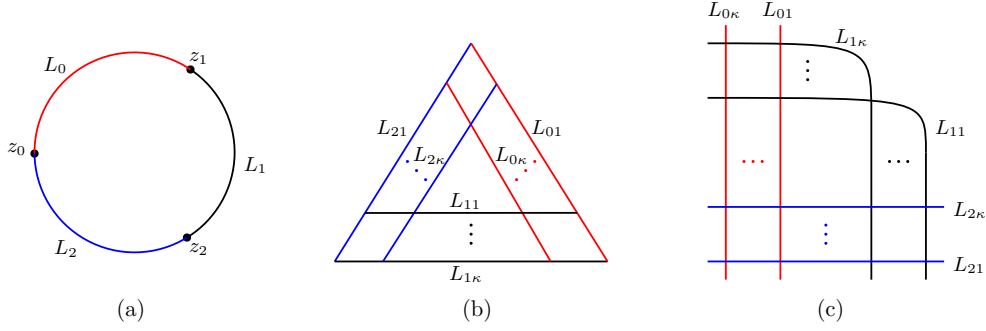}
     \caption{(a) $D_3$, the $A_\infty$ base direction; (b) The Lagrangians in the $T^*I_1$ direction; (c) The Lagrangians in the $T^*I_2$ direction.}
    \label{figure-X}
\end{figure}

For $i=0,1,2$, let $\boldsymbol{y}_i=\{y_{i1},\dots,y_{i\kappa}\}$ be a tuple of intersection points $y_{ij}\in \widehat{L}_{(i-1)j}\cap \widehat{L}_{ij'}$, where $\{1',\dots,\kappa'\}$ is some permutation of $\{1,\dots,\kappa\}$.
Let $J$ be a small generic perturbation of $J_{D_3}\times J_1\times J_2$, where $J_{D_3},J_1,J_2$ are the standard complex structures on $D_3,T^*I_1, T^*I_2$, viewed as subsets of $\mathbb{C}$.
Let $\mathcal{M}_{J}(\boldsymbol{y}_1,\boldsymbol{y}_2,\boldsymbol{y}_0)$ be the moduli space of maps
\begin{equation}
    u:(\dot F,j)\to(\widehat{X},J),
\end{equation}
where $(F,j)$ is a compact Riemann surface with boundary, $\boldsymbol{p}_i$ are disjoint tuples of boundary punctures of $F$ for $i=0,1,2$, and $\dot F=F\backslash\cup_i \boldsymbol{p}_i$. The map $u$ satisfies
\begin{align}
    \label{floer-condition}
    \left\{
        \begin{array}{ll}
            \text{$du\circ j=J\circ du$;}\\
            \text{each component of $\partial \dot F$ is mapped to a unique $\widehat{L}_{ij}$;}\\
            \text{$\pi_X\circ u$ tends to $\boldsymbol{y}_i$ as $s_i\to+\infty$ for $i=1,\dots,m$;}\\
            \text{$\pi_X\circ u$ tends to $\boldsymbol{y}_0$ as $s_0\to-\infty$;}\\
            \text{$\pi_{D_3}\circ u$ is a $\kappa$-fold branched cover of $D_3$,}
        \end{array}
    \right.
\end{align}
where the 3rd condition means that $\pi_X\circ u$ maps the neighborhoods of the punctures of $\boldsymbol{p}_i$ asymptotically to the Reeb chords of $\boldsymbol{y}_i$ for $i=1,\dots,m$ at the positive ends. The 4th condition is similar. 

The $m_2$-composition map of $\operatorname{End}(L^{\otimes\kappa})$ is then defined as
\begin{equation}
\label{m_2}
    m_2(\boldsymbol{y_1},\boldsymbol{y_2})=\sum_{\boldsymbol{y_0},\chi\leq\kappa}\#\mathcal{M}^{\mathrm{ind}=0,\chi}_{J}(\boldsymbol{y_1},\boldsymbol{y_2},\boldsymbol{y_0})\cdot\hbar^{\kappa-\chi}\cdot\boldsymbol{y_0}.
\end{equation}
where the superscript ``$\mathrm{ind}$'' denotes the Fredholm index and ``$\chi$'' denotes the Euler characteristic of $F$;  the symbol $\#$ denotes the signed count of the corresponding moduli space.


It is known that a choice of spin structures on the Lagrangians determines a canonical orientation of the moduli space. 
The Lagrangian in our case is the cotangent fiber which is topologically $\mathbb{R}^2$. So there is a unique spin structure.   
We omit the details about the orientation, and refer the reader to \cite[Section 3]{colin2020applications}.

\section{The case of $\kappa=2$}
\label{section-2}

In this section we compute $\operatorname{End}(L^{\otimes 2})$ as a model case. 
The general case will be discussed in Section \ref{section-3}.

For $0 \le i < j \le 2$, there are two Floer generators of $CF^*(\mathcal{L}_i,\mathcal{L}_j)$: $T_{{\operatorname{id}}}$ and $T_1$, where $T_{{\operatorname{id}}}=(q_1,q_2)$ with $q_1\in L_{i1}\cap L_{j1}$ and $q_2\in L_{i2}\cap L_{j2}$, and $T_1=(q_1,q_2)$ with $q_1\in L_{i1}\cap L_{j2}$ and $q_2\in L_{i2}\cap L_{j1}$. 
The main result of this section is the following.

\begin{proposition}
The multiplication on $\operatorname{End}(L^{\otimes 2})$ is given by
\begin{align*}
T_{\operatorname{id}}\cdot T_{\operatorname{id}}&=T_{\operatorname{id}}; \\
T_{\operatorname{id}}\cdot T_{1}&=T_{1}; \\
 T_{1} \cdot T_{\operatorname{id}}&=T_{1}; \\
T_1 \cdot T_1&=1+\hbar T_1.
\end{align*}
Hence, Theorem \ref{main-thm} holds for $\kappa=2$.
    \label{prop-2}
\end{proposition}

The proof of this proposition occupies the rest of the section. 
We directly compute the moduli spaces. 
There are trivial curves with $\chi=2$ accounting for the $\hbar^0$ terms in the multiplication. 
We show that $\mathcal{M}^{\chi<2}_J(\boldsymbol{y}_1,\boldsymbol{y}_2,\boldsymbol{y}_0)=\emptyset$ for almost all cases except that $\mathcal{M}^{\chi=1}_J(T_1,T_1,T_1) \neq \emptyset$ accounting for the $\hbar^1$ term in $T_1\cdot T_1$.  
The main strategy to prove the nonexistence of curves is to stretch the Lagrangians in the $T^*I_1$-direction and apply the Gromov compactness.

For later use, we make the following conventions: 
\begin{itemize}
    \item We denote the length of the line segment of $L_{1\kappa}$ in the $I_1$-direction by $d$, see Figure \ref{figure-X}(b);
    \item For $q\in X$, we denote its projection in the $T^*I_1$ (resp. $T^*I_2$)-direction by $q'$ (resp. $q''$).
    \item We denote the line segment between $q_1$ and $q_2$ by $(q_1q_2)$.
    \item When plotting figures, we denote the intersections $q_i$ by $i$.
    \item When taking the limit, we denote the degenerated domain by $\dot{F}'$ and its irreducible component containing $\{p_1,p_2,\dots\}$ by $\dot{F}'_{(12\dots)}$.
    
\end{itemize}

\begin{lemma} \label{lemma 11=1}
    $T_{\operatorname{id}}\cdot T_{\operatorname{id}}=T_{\operatorname{id}}$.
\end{lemma}
\begin{proof}
    We first show that
    \begin{equation}
        \#\mathcal{M}^{\chi}_{J}(T_{\operatorname{id}},T_{\operatorname{id}},T_{\operatorname{id}})=
        \left\{
            \begin{array}{lr}
            1,\,\,\chi=2 &  \\
            0,\,\,\chi<2 &  
            \end{array}.
        \right.
    \end{equation}
   The Floer generators are shown in Figure \ref{fig-1-1-1}.
    
    \begin{figure}
    \centering
    \includegraphics[width=11cm]{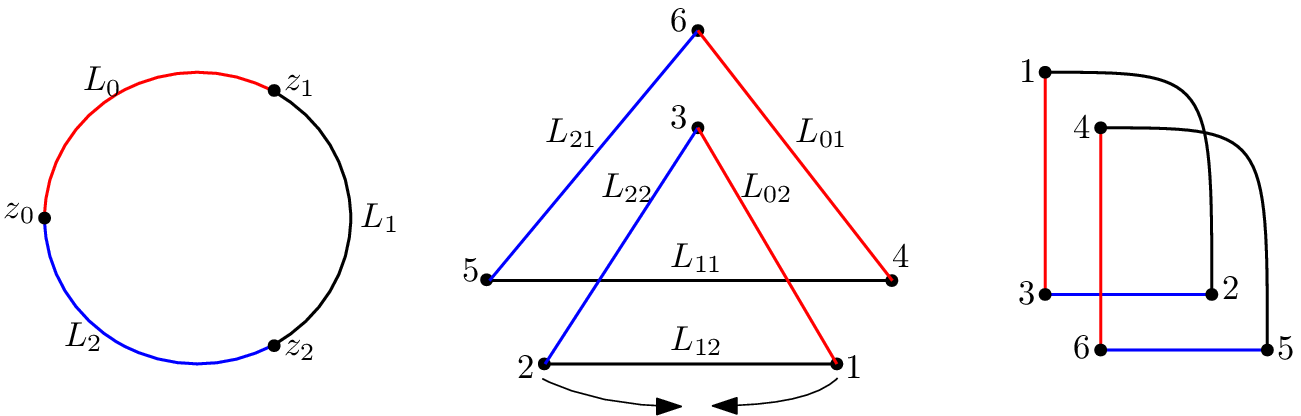}
    \caption{Generators for $\mathcal{M}_{J}(T_{\operatorname{id}},T_{\operatorname{id}},T_{\operatorname{id}})$.}
    \label{fig-1-1-1}
\end{figure}
    
    If $\chi=2$, there is a unique trivial holomorphic curve consisting of two disks. So $\#\mathcal{M}^{\chi=2}_{J}(T_{\operatorname{id}},T_{\operatorname{id}},T_{\operatorname{id}})=1$. 
    
    If $\chi<2$, let $d\to0$, i.e., let $q'_1$ and $q'_2$ get closer. 
    In the limit, since there is no slit or branch points separating $q'_1$ and $q'_2$, $\dot F'_{(12)}$ bubbles off as a triangle with vertices $\{p_1,p_2,p_a\}$, where $\{p_a\}$ is a boundary nodal point.
    The projection of $\dot F'_{(12)}$ under $\pi_{T^*I_2}\circ u$ is a homeomorphism to the triangle with vertices $\{q''_1,q''_2,q''_3\}$. Hence the projection of $\dot F'_{(3)}$ under $\pi_{T^*I_2}\circ u$ is a constant map to $q''_3$.
    Since $\pi_{T^*I_2}\circ u$ is of degree 0 or 1 near $q''_3$, the image $\pi_{T^*I_2}\circ u(\dot{F}\backslash(\dot{F}'_{(12)}\cup \dot{F}'_{(3)}))$ is disjoint from $q''_3$. It follows that $\dot{F}'_{(12)}\cup \dot{F}'_{(3)}$ is a connected component of $\dot{F}'$.
    Therefore, $\dot{F}$ consists of two components before the degeneration, which are homeomorphically mapped to the triangles $\{q''_1,q''_2,q''_3\}$ and $\{q''_4,q''_5,q''_6\}$ under $\pi_{T^*I_2}\circ u$, respectively. 
    So $\chi=2$, which is a contradiction.  We conclude that $ \#\mathcal{M}^{\chi}_{J}(T_{\operatorname{id}},T_{\operatorname{id}},T_{\operatorname{id}})=0$ if $\chi<2$.
    
    \begin{figure}
        \centering
        \includegraphics[width=11cm]{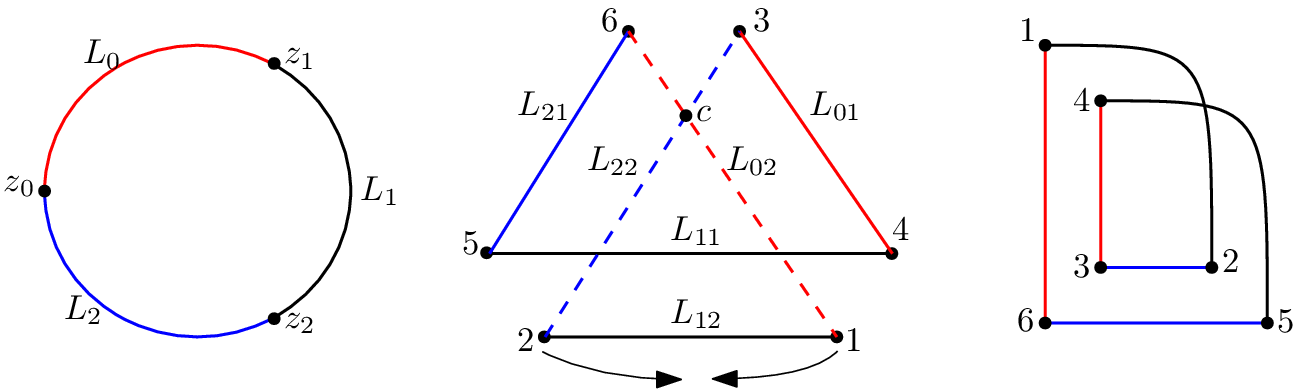}
        \caption{Generators for $\mathcal{M}_{J}(T_{\operatorname{id}},T_{\operatorname{id}},T_1)$.}
        \label{fig-1-1-x}
    \end{figure}
    

    We next show that $\#\mathcal{M}_{J}(T_{\operatorname{id}},T_{\operatorname{id}},T_1)=0$. The generators are shown in Figure \ref{fig-1-1-x}. 
    As $d\to0$, $\dot{F}'_{(12)}$ bubbles off as a triangle with vertices $\{p_1,p_2,p_a\}$, where $p_a\in\dot{F}'$ is a nodal point.
    Denote the union of irreducible components of $\dot{F}'$ containing the preimage of the dashed lines in the $T^*I_1$-direction by $\dot{F}'_{\text{dash}}$.
    Since $p_3$, $p_6$, and $p_a$ are mapped to $z_0$ under $\pi_{D_3}\circ u$ in the limit, the preimages of the dashed lines are also mapped to $z_0$. Hence, $\dot{F}'_{\text{dash}}$ is mapped to the constant point $z_0$ under $\pi_{D_3}\circ u$.
    Since $(q'_5q'_6)$ cannot be separated by slits, we have $q'_5\in\dot{F}'_{\text{dash}}$ and $\pi_{D_3}\circ u(q'_5)=z_0$.
    This contradicts with the fact that $\pi_{D_3}\circ u(q'_5)=z_2$.
    Therefore, $\mathcal{M}_{J}(T_{\operatorname{id}},T_{\operatorname{id}},T_1)=\emptyset$.
    \end{proof}
    

\begin{lemma} \label{lemma 1T=T}
    $T_{\operatorname{id}}\cdot T_1=T_1$.
    \label{lemma-1-x}
\end{lemma}
\begin{proof}
    
    \begin{figure}
        \centering
        \includegraphics[width=11cm]{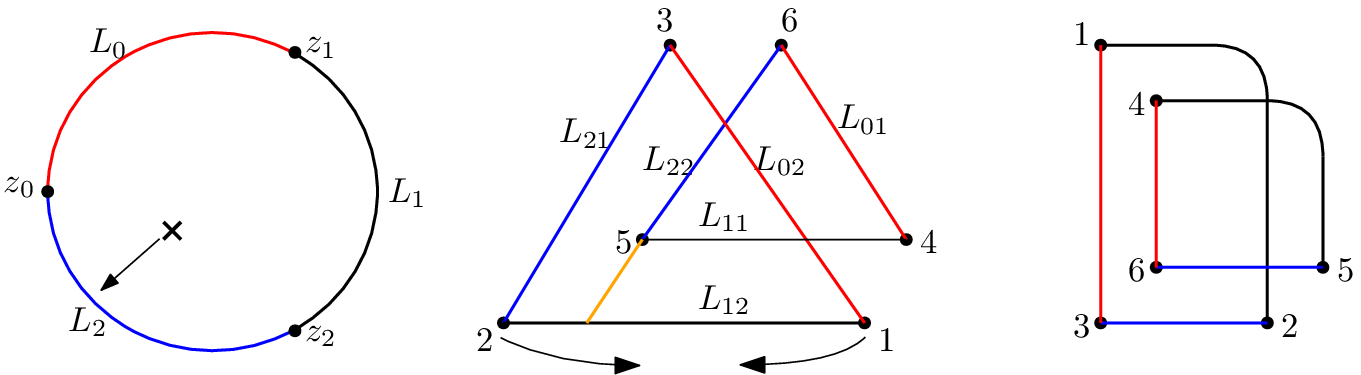}
        \caption{Generators for $\mathcal{M}_{J}(T_{\operatorname{id}},T_1,T_1)$.}
        \label{fig-1-x-x}
    \end{figure}
    
    First we show that
    \begin{equation}
        \#\mathcal{M}^{\chi}_{J}(T_{\operatorname{id}},T_1,T_1)=
        \left\{
            \begin{array}{lr}
            1,\,\,\chi=2 &  \\
            0,\,\,\chi<2 &  
            \end{array}.
        \right.
    \end{equation}
    The generators are shown in Figure \ref{fig-1-x-x}. 
    
    If $\chi=2$, there is a unique trivial holomorphic curve consisting of two disks. So $\#\mathcal{M}^{\chi=2}_{J}(T_{\operatorname{id}},T_1,T_1)=1$.
    
    If $\chi<2$, as $d\to0$, there are two cases: 
    \begin{itemize}
        \item If the orange slit extending $(q'_6q'_5)$ is not long, then $\dot{F}'_{(12)}$ bubbles off as a triangle with vertices $\{p_1,p_2,p_a\}$. 
        The remaining proof is the same as that of $\#\mathcal{M}^{\chi}_{J}(T_{\operatorname{id}},T_{\operatorname{id}},T_{\operatorname{id}})=0$ for $\chi<2$ in Lemma \ref{lemma 11=1}.
        Hence, the limiting curve does not exist.
        \item If the orange slit extending $(q'_6q'_5)$ is long, then there is a branch point approaching the interior of $L_2$ in the $D_3$-direction (as in the left of Figure \ref{fig-1-x-x}). 
        In the limit, the preimage of the branch point on the domain tends to some nodal point $p_n$ so that $\pi_{D_3}\circ u(p_n)\in L_2$. It implies that $\pi_{T^*I_2}\circ u(p_n)\in L_{21}\cap L_{22}$.
        This contradicts with the fact that $L_{21}\cap L_{22}=\emptyset$ in the $T^*I_2$-direction.
    \end{itemize}

    \begin{figure}
        \centering
        \includegraphics[width=11cm]{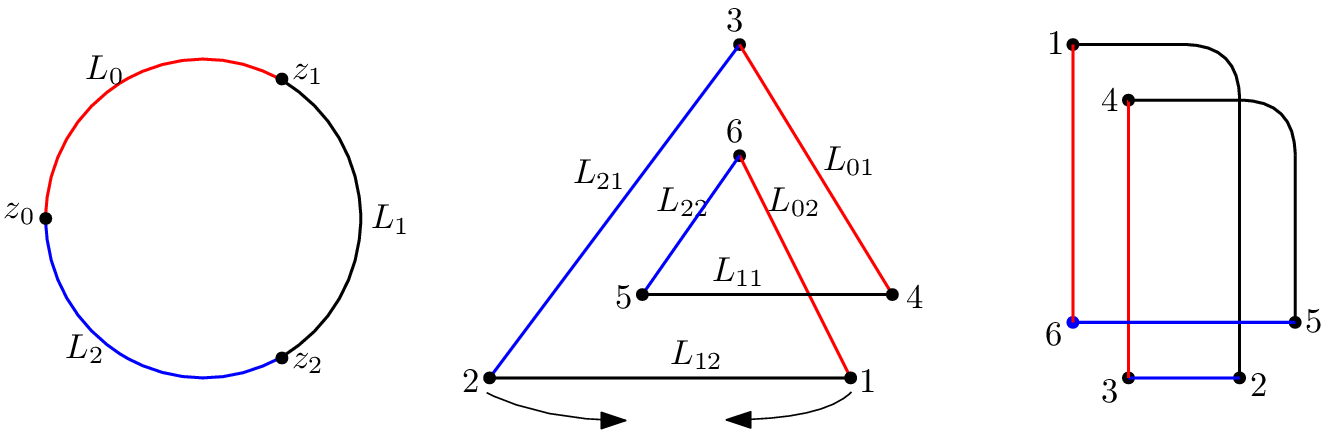}
        \caption{Generators for $\mathcal{M}_{J}(T_{\operatorname{id}},T_1,T_{\operatorname{id}})$.}
        \label{fig-1-x-1}
    \end{figure}

    Next we show that
       $ \#\mathcal{M}^{\chi}_{J}(T_{\operatorname{id}},T_1,T_{\operatorname{id}})=0$,
    for $\chi\leq 2$. 
    The generators are shown in Figure \ref{fig-1-x-1}. 
    As $d\to0$, $\dot F'_{(12)}$ bubbles off as a triangle $\{p_1,p_2,p_a\}$.
    Then $p_a$ should be mapped to the intersection of the extension of the line segments $(q''_1q''_6)$ and $(q''_2q''_3)$. But this is impossible since the degree of the projection $\pi_{T^*I_2}\circ u$ is 0 near the intersection. 
\end{proof}

The similar arguments will be used in the proofs of Propositions \ref{prop-key1} and \ref{prop-key2}. 
In general, $\dot{F}'_{(12)}$ always bubbles off as a triangle as $d\to0$. Here, $q'_1,q'_2$ are on the bottom Lagrangian $L_{1\kappa}$ in the $T^*I_1$-direction.
We then analyze the remaining irreducible components of $\dot{F}'$ and reduce the problem to simpler cases.

\begin{lemma} \label{lemma T1=T}
    $T_1\cdot T_{\operatorname{id}}=T_1$.
\label{lemma-x-1}
\end{lemma}
\begin{proof}
    This is similar to the proof of Lemma \ref{lemma-1-x}.
   \end{proof}

\begin{figure}
    \centering
    \includegraphics[width=13cm]{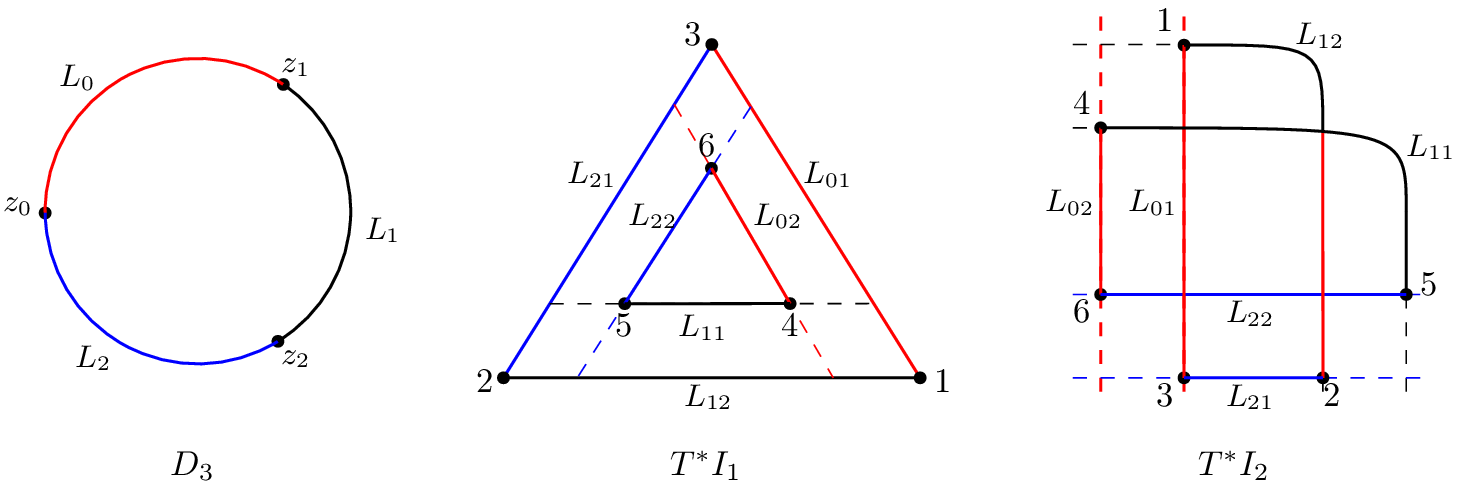}
    \caption{Generators for $\mathcal{M}_{J}(T_1,T_1,T_{\operatorname{id}})$.}
    \label{fig-k-2-chi-0}
\end{figure}


\begin{lemma}     \label{lemma TT}
    $T_1\cdot T_1=T_{\operatorname{id}}+\hbar T_1$.
\end{lemma}
\begin{proof}
    We first show that
    \begin{equation}
        \#\mathcal{M}^{\chi}_{J}(T_1,T_1,T_{\operatorname{id}})=
        \left\{
            \begin{array}{lr}
                1,\,\,\chi=2 &  \\
                0,\,\,\chi<2 &  
            \end{array}.
        \right.
    \end{equation}
    The generators are shown in Figure \ref{fig-k-2-chi-0}. 
    
    If $\chi=2$, then there is a unique trivial holomorphic curve consisting of two disks, so $\#\mathcal{M}^{\chi=2}_{J}(T_1,T_1,T_{\operatorname{id}})=1$.
    
    If $\chi<2$, then $\dot F'_{(12)}$ bubbles off as a triangle as $d\to0$. 
    The projection of $\dot F'_{(12)}$ under $\pi_{T^*I_2}\circ u$ is a homeomorphism to the triangle $\{q''_1,q''_2,q''_3\}$.
    Consequently, the projection of $\dot F'_{(3)}$ under $\pi_{T^*I_2}\circ u$ is the constant map to $q''_3$.
    Since $\pi_{T^*I_2}\circ u$ is of degree 0 or 1 near $q''_3$, the image of $\pi_{T^*I_2}\circ u(\dot{F}\backslash(\dot{F}'_{(12)}\cup \dot{F}'_{(3)}))$ is disjoint from $q''_3$. 
    It follows that $\dot{F}'_{(12)}\cup \dot{F}'_{(3)}$ is a connected component of $\dot{F}'$.
    Therefore, $\dot{F}$ consists of two components before the degeneration, which are homeomorphically mapped to the triangles $\{q''_1,q''_2,q''_3\}$ and $\{q''_4,q''_5,q''_6\}$ under $\pi_{T^*I_2}\circ u$, respectively. 
    So $\chi=2$, which is a contradiction.
    ~\\

    \begin{figure}
        \centering
        \includegraphics[width=13cm]{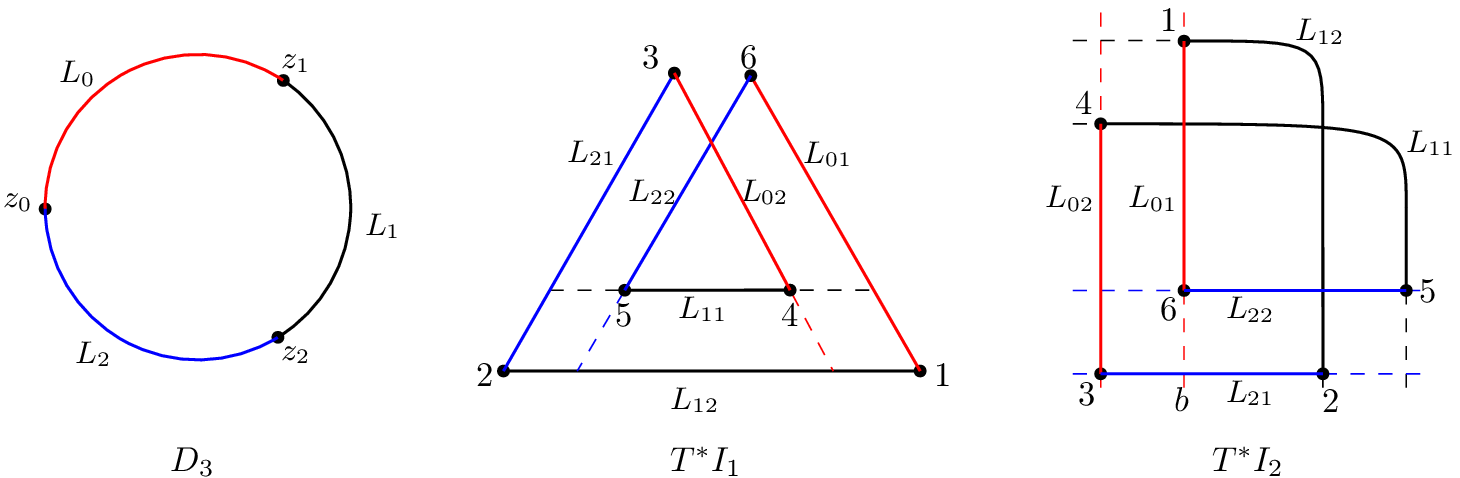}
        \caption{Generators for $\mathcal{M}_{J}(T_1,T_1,T_1)$.}
        \label{fig-curve-1}
    \end{figure}

    \begin{figure}
        \centering
        \includegraphics[height=3cm]{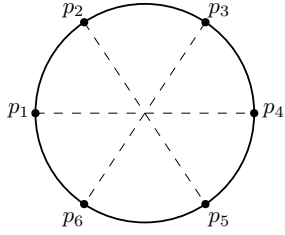}
        \caption{A disk $\dot F=D_6$ which satisfies the involution condition.}
        \label{fig-disk}
    \end{figure}
    
    We next show that
    \begin{equation}
        \#\mathcal{M}^{\chi}_{J}(T_1,T_1,T_1)=
        \left\{
            \begin{array}{lr}
            1,\,\,\chi=1 &  \\
            0,\,\,\chi\neq1 &  
            \end{array}.
        \right.
    \end{equation}
    The generators are shown in Figure \ref{fig-curve-1}. 
    
    We denote the moduli space of domain $(\dot{F},j)$ by $\mathcal{M}(\dot{F})$. By Riemann-Roch formula, $\mathrm{dim}\,\mathcal{M}(\dot{F})=3(\kappa-\chi)$. 
    Consider the moduli space of pseudoholomorphic maps from $(\dot{F},j)$ to each direction $D_3$, $T^*I_1$ and $T^*I_2$, denoted by $\mathcal{M}(D_3)$, $\mathcal{M}({T^*I_1})$ and $\mathcal{M}({T^*I_2})$, respectively. The index formula says
    \begin{equation*}
        \mathrm{dim}\,\mathcal{M}(D_3)=\mathrm{dim}\,\mathcal{M}({T^*I_1})=\mathrm{dim}\,\mathcal{M}({T^*I_2})=2(\kappa-\chi),
    \end{equation*}
    for generic $J$. 
We have
    \begin{equation}
        \mathcal{M}_{J}(T_1,T_1,T_1)=\mathcal{M}(D_3)\cap\mathcal{M}({T^*I_1})\cap\mathcal{M}({T^*I_2}),
    \end{equation}
This is our main strategy to count curves in $\mathcal{M}_{J}(T_1,T_1,T_1)$: we compute the moduli space for each direction and then count their intersection number.
    
    The moduli space of curves restricted to each direction has an explicit parametrization. For example, $\pi_{D_3}\circ u$ from $\dot{F}$ to $D_3$ is a $\kappa$-fold branched cover, and its restriction to $\partial \dot{F}$ is a $\kappa$-fold cover over $S^1$. 
    Generically, $\pi_{D_3}\circ u$ is parametrized by the positions of $\kappa-\chi$ double branch points on $\dot{F}$ over $D_3$.

    In the case $\kappa=2$ and $\chi=1$, $\dot F=D_6$ is a disk with 6 boundary punctures. 
    The moduli space of $(\dot F,j)$ is 
    \begin{equation*}
        \mathcal{M}(\dot{F})\simeq \mathbb{R}^3.
    \end{equation*}
    Then we consider the cut-out moduli space $\mathcal{M}(D_3)$, viewed as a subset of $\mathcal{M}(\dot{F})$. 
    The deck transformation of $\pi_{D_3}\circ u$ imposes an involution condition on $\dot F$. In other words, we require that $\{p_i,p_{i+3}\}$ lie on a diameter for $i=1,2,3$ after some fractional linear transformation. 
    Therefore,
    \begin{equation*}
        \mathcal{M}(D_3)\simeq \mathbb{R}^2.
    \end{equation*}
The moduli spacce $\mathcal{M}(D_3)$ admits a compactification $\overline{\mathcal{M}}(D_3)$, which is described in Figure \ref{fig-disk-moduli}.
    
    \begin{figure}
        \centering
        \includegraphics[height=7cm]{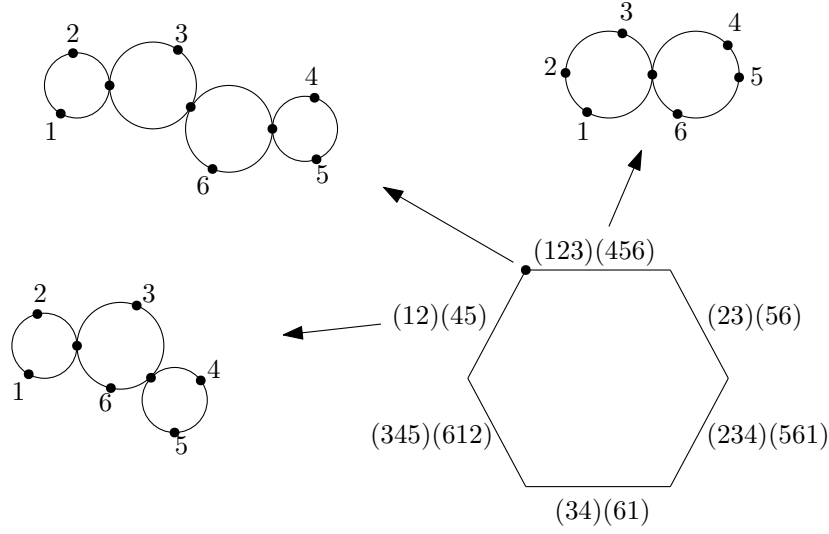}
        \caption{The compactified moduli space $\overline{\mathcal{M}}(D_3)$ is described by the hexagon. Index $i$ stands for $p_i\in\partial \dot F$. The indices inside brackets describe the bubbling behavior, e.g., $(12)(45)$ means $p_1$ and $p_2$ (resp. $p_4$ and $p_5$) are close to each other. The involution condition is preserved on boundary strata, e.g., the cross ratio of the two bubble disks on the stratum $(123)(456)$ are the same.}
        \label{fig-disk-moduli}
    \end{figure}
    
    We first consider $\partial\mathcal{M}(D_3)\cap\mathcal{M}(T^*I_1)$. 
    The map from a disk to the middle of Figure \ref{fig-curve-1} may have a double branch point inside the inner region with degree 2. 
    As the branch point touches the boundary of the inner region, it is replaced by a slit with 2 switch points along the Lagrangians. 
    Since we are interested in $\partial\mathcal{M}(D_3)$, the bubbling behavior in Figure \ref{fig-disk-moduli} requires the slit to be very long so that some switch point meets another Lagrangian. 
    The involution condition further requires that such switch points come in pairs. 
    We conclude that $\partial\mathcal{M}(D_3)\cap\mathcal{M}(T^*I_1)$ consists of two points : one passes $q'_5$ to its left extending the line segment $(q'_4q'_5)$ and  downwards extending $(q'_6q'_5)$; the other passes $p_4$ to its right extending $(q'_5q'_4)$ and downwards extending $(q'_3q'_4)$. The two points are decipted as the violet dots in Figure \ref{fig-curve-1-sol}.
    
    \begin{figure}
        \centering
        \includegraphics[height=7cm]{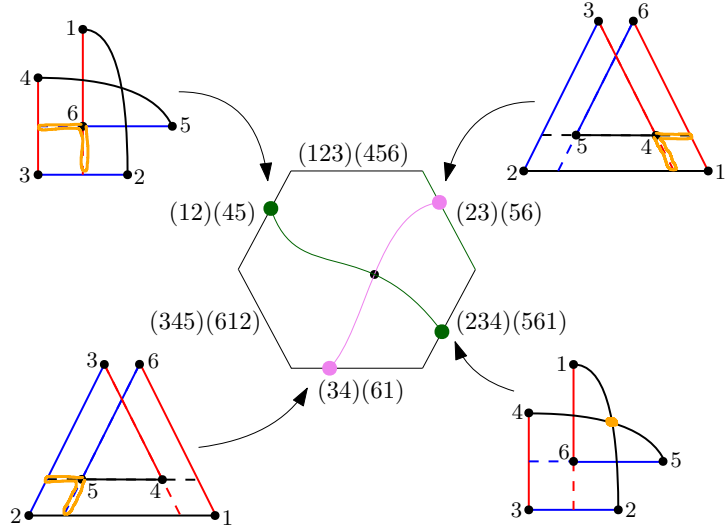}
        \caption{The orange curves in the pictures outside the hexagon represent slits. The black dot inside the hexagon is a curve in $\mathcal{M}_{J}(D_3)\cap\mathcal{M}_{J}(T^*I_1)\cap\mathcal{M}_{J}(T^*I_2)$.}
        \label{fig-curve-1-sol}
    \end{figure}
    
    For $\partial\mathcal{M}(D_3)\cap\mathcal{M}(T^*I_2)$, consider the right of Figure \ref{fig-curve-1}. 
    Similar to the previous paragraph, the degeneration of $D_6$ requires the existence of long slits. 
    There are two curves: one with a slit passing $q''_6$ to its left and downwards; the other lies on the Lagrangian $(q''_1q''_2)$ or $(q''_4q''_5)$ with one switch point meeting the intersection point $(q''_1q''_2)\cap(q''_4q''_5)$. The two curves are decipted as the dark-green dots in Figure \ref{fig-curve-1-sol}. 
    
    The violet and dark-green points in $\partial\mathcal{M}(D_3)$ indicate that $\mathcal{M}(D_3)\cap\mathcal{M}(T^*I_1)$ and $\mathcal{M}(D_3)\cap\mathcal{M}(T^*I_2)$ have intersection of algebraic count 1 inside $\mathcal{M}(D_3)$. Thus,
    \begin{equation} \label{eq TThT}
        \#\mathcal{M}^{\chi=1}_{J}(T_1,T_1,T_1)=\#\mathcal{M}(D_3)\cap\mathcal{M}(T^*I_1)\cap\mathcal{M}(T^*I_2)=1.
    \end{equation}
    
If $\chi\neq1$, we show that $\#\mathcal{M}^{\chi}_{J}(T_1,T_1,T_1)=0$.
    As $d\to0$, $\dot F'_{(12)}$ bubbles off as a triangle.
    Since the projection of $\dot{F}'\backslash\dot F'_{(12)}$ to $T^*I_2$ is of degree 1 to its image (the polygon composed of $\{q''_3,q''_4,q''_5,q''_6,q''_b\}$ in Figure \ref{fig-curve-1}), the domain before degeneration has to be a disk. This contradicts with the fact that $\chi\neq1$.
\end{proof}

The counting in (\ref{eq TThT}) is essentially the only case in our direct computation where a nontrivial curve exists. It corresponds to the deformation $\tilde{T}_{i}^2=1+\hbar \tilde{T}_i$ from the symmetric group to the Hecke algebra.

\section{The general case}
\label{section-3}

In this section, we compute $\operatorname{End}(L^{\otimes \kappa})$ by induction on $\kappa$. 
Recall that $\operatorname{End}(L^{\otimes \kappa})$ is freely generated by $T_w=\{y_1,\cdots,y_{\kappa}\}$, where $y_j \in L_{0j}\cap L_{1w(j)}$ and $w \in S_{\kappa}$ is viewed as a permutation.  
We compute $T_{w_1}\cdot T_{w_2}$ for $w_1, w_2 \in S_{\kappa}$ by a case-by-case discussion depending on how $w_1$ acts on the last one or two elements of $\{1,\dots,\kappa\}$.

The first case is when $w_1$ fixes the last element.     
The schematic picture is shown in Figure \ref{fig-induction-1}.
The following proposition is a generalization of Lemmas \ref{lemma 11=1}, \ref{lemma 1T=T}.

\begin{figure}
    \centering
    \includegraphics[width=10cm]{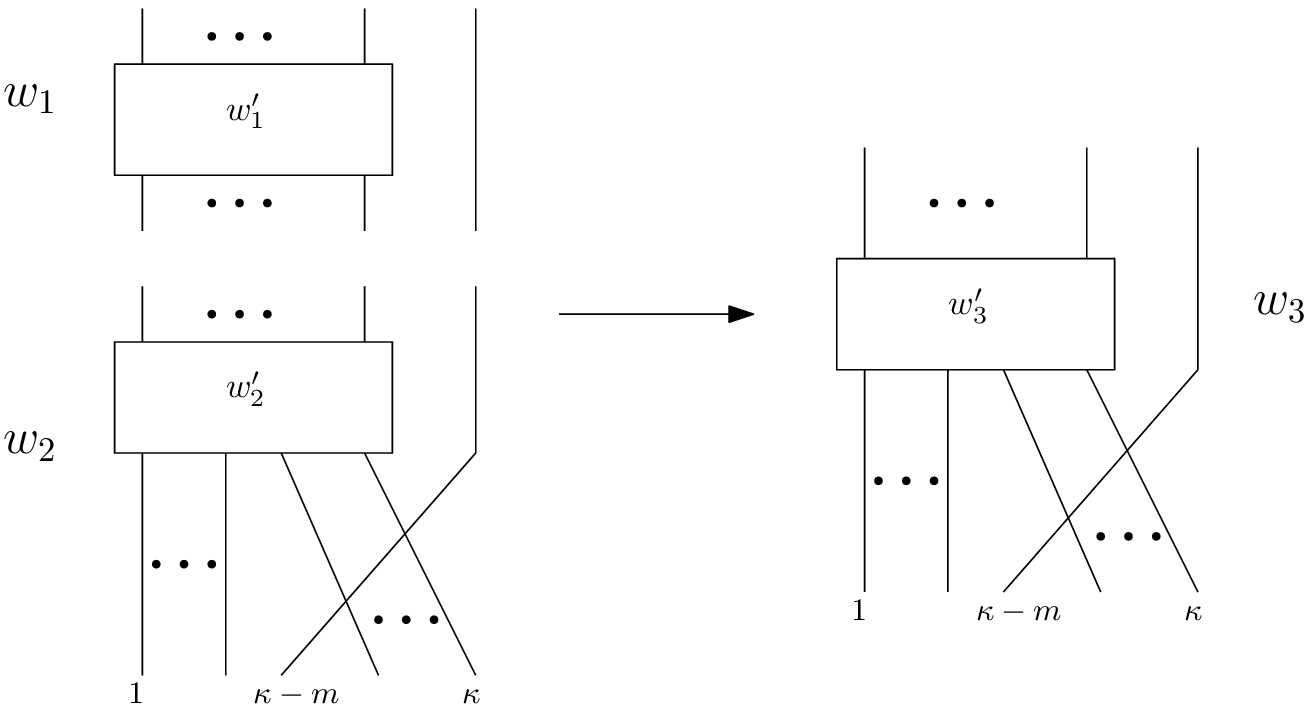}
    \caption{The case for Proposition \ref{prop-key1}.}
    \label{fig-induction-1}
\end{figure}

\begin{proposition} 
\label{prop-key1}
    For $w_1,w_2,w_3 \in S_{\kappa}$, suppose $w_1=w_1', w_2=w_2's_{\kappa-1}s_{\kappa-2}\cdots s_{\kappa-m}$, where $w_1', w_2' \in S_{\kappa-1}$ and $m \ge 0$.
    We have 
    \begin{align*} 
    \#\mathcal{M}^{\chi}(T_{w_1}, T_{w_2}, T_{w_3}) =\left\{
            \begin{array}{ll}
            \#\mathcal{M}^{\chi-1}(T_{w_1'}, T_{w_2'}, T_{w_3'}) & \mathrm{if}\,\,\, w_3=w_3's_{\kappa-1}s_{\kappa-2}\cdots s_{\kappa-m},\\
                   0 & \mathrm{otherwise},
                   \end{array}\right.
    \end{align*}
    where $w_3' \in S_{\kappa-1}$.
\end{proposition}

\begin{proof}
    Suppose that the strand of $w_3$ starting from the position $\kappa$ ends on the position $t$. See the right of Figure \ref{fig-induction-1}. Here, we are reading the picture for $w_3$ from top to bottom.  
    We consider the following three cases depending on $t$.
\begin{enumerate}
\item $t=\kappa-m$.    
Let $w_3=w_3's_{\kappa-1}s_{\kappa-2}\cdots s_{\kappa-m}$ for $w_3' \in S_{\kappa-1}$.
This case is shown in Figure \ref{fig-induction-1-1}.
    The last vertical strand of $w_1$ in Figure \ref{fig-induction-1} corresponds to $p_1$ in Figure \ref{fig-induction-1-1}.
    As $d\to0$, $\dot{F}'_{(12)}$ bubbles off as a triangle with vertices $\{p_1,p_2,p_a\}$. 
    The image of $\dot{F}'_{(12)}$ in the $T^*I_2$-direction is the orange triangle and the image of $\dot{F}'_{(3)}$ is the constant point $q''_3$.
    Since $\pi_{T^*I_2}\circ u$ is of degree 0 or 1 near $q''_3$, the image $\pi_{T^*I_2}\circ u(\dot{F}\backslash(\dot{F}'_{(12)}\cup\dot{F}'_{(3)}))$ is disjoint from $q''_3$. It implies that $\dot{F}'_{(12)}\cup\dot{F}'_{(3)}$ is a connected component of $\dot{F}'$. 
    Therefore, $\dot{F}_{(123)}$ is a connected component of $\dot{F}$ before the degeneration, and it is mapped homeomorphically to the triangle $\{q''_1,q''_2,q''_3\}$ under $\pi_{T^*I_2}\circ u$.
    After removing the component $\dot{F}_{(123)}$, we see that $\#\mathcal{M}^{\chi}(T_{w_1}, T_{w_2}, T_{w_3})=\#\mathcal{M}^{\chi-1}(T_{w'_1}, T_{w'_2}, T_{w'_3})$.
        \begin{figure}
        \centering
        \includegraphics[width=12cm]{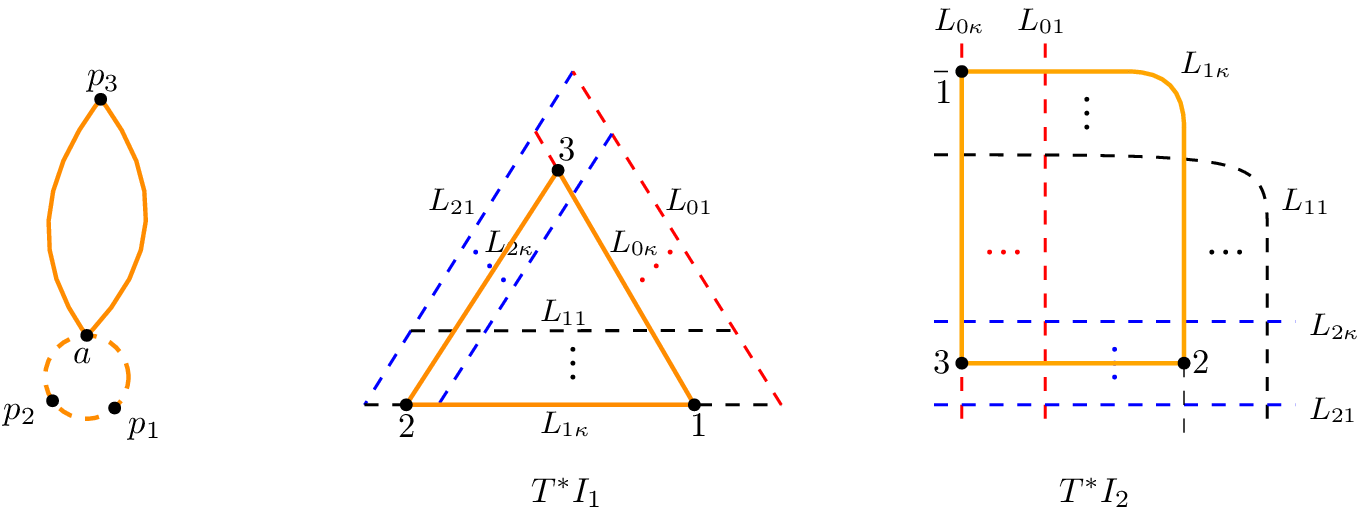}
        \caption{The case $t=\kappa-m$.}
        \label{fig-induction-1-1}
    \end{figure}

\item $t>\kappa-m$. 
    This case is shown in Figure \ref{fig-induction-1-2}.
    As $d\to0$, $\dot{F}'_{(12)}$ bubbles off as a triangle with vertices $\{p_1,p_2,p_a\}$. 
    There is a vertex $p_b$ in the component $\dot{F}'_{(3)}$ which is adjacent to $p_a$.
    Since $\pi_{T^*I_2}\circ u$ has degree 0 near the intersection between the extensions of $(q''_1q''_b)$ and $(q''_2q''_3)$, $\dot{F}'_{(12)}$ cannot be a triangle. This leads to a contradiction.
    Therefore, $\#\mathcal{M}^{\chi}(T_{w_1}, T_{w_2}, T_{w_3})=0$.   
    \begin{figure}
        \centering
        \includegraphics[width=12cm]{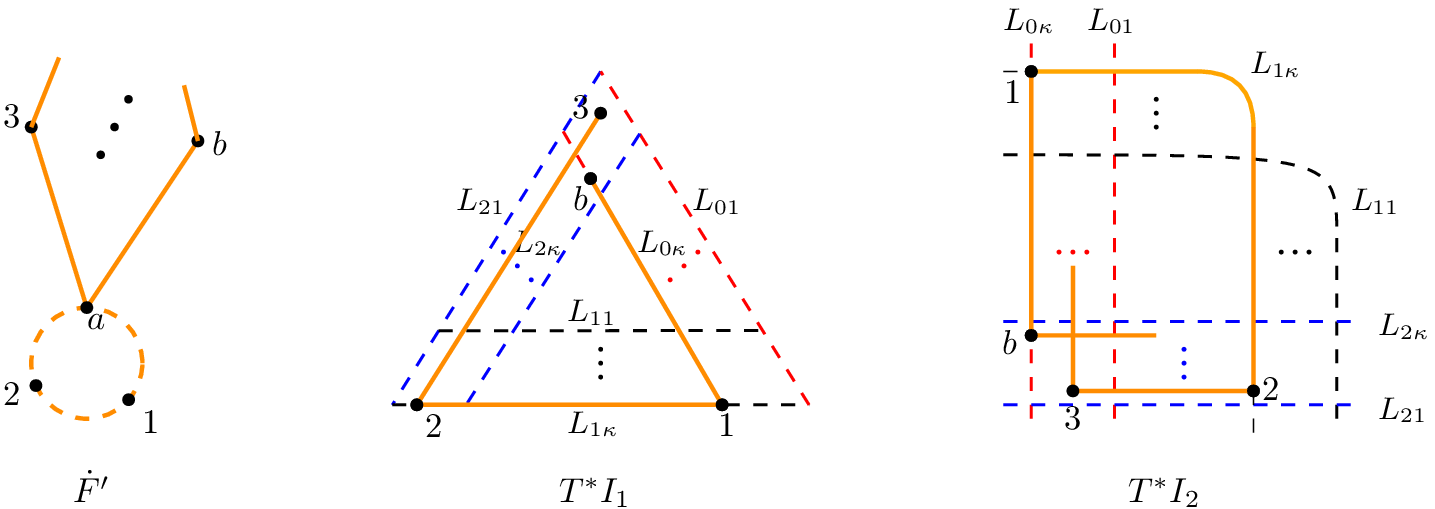}
        \caption{The case $t>\kappa-m$.}
        \label{fig-induction-1-2}
    \end{figure}  
\item $t<\kappa-m$. 
    This case is shown in Figure \ref{fig-induction-1-3}.
    As $d\to0$, $\dot{F}'_{(12)}$ bubbles off as a triangle with vertices $\{p_1,p_2,p_a\}$.
    On one hand, similar to the proof of $\#\mathcal{M}_{J}(T_{\operatorname{id}},T_{\operatorname{id}},T_1)=0$ of Lemma \ref{lemma 11=1}, the projection of $\dot{F}'_{(b)}$ under $\pi_{D_3}\circ u$ is the constant map to $z_0$.
    On the other hand, the line denoted by the orange arrow is disjoint from $L_{0i}$ for $i=1,\dots,\kappa-1$ since $(q'_1q'_b)$ lies on $L_{0\kappa}$. So $q'_b$ cannot be separated from the bottom-left region. But the generators in this region are mapped to $z_2$ in the $D_3$-direction. 
    We conclude that $\pi_{D_3}\circ u(\dot{F}'_{(b)})$ cannot be far from $z_2$. This is a contradiction. 
    Therefore, $\#\mathcal{M}^{\chi}(T_{w_1}, T_{w_2}, T_{w_3})=0$.
      \begin{figure}
        \centering
        \includegraphics[width=12cm]{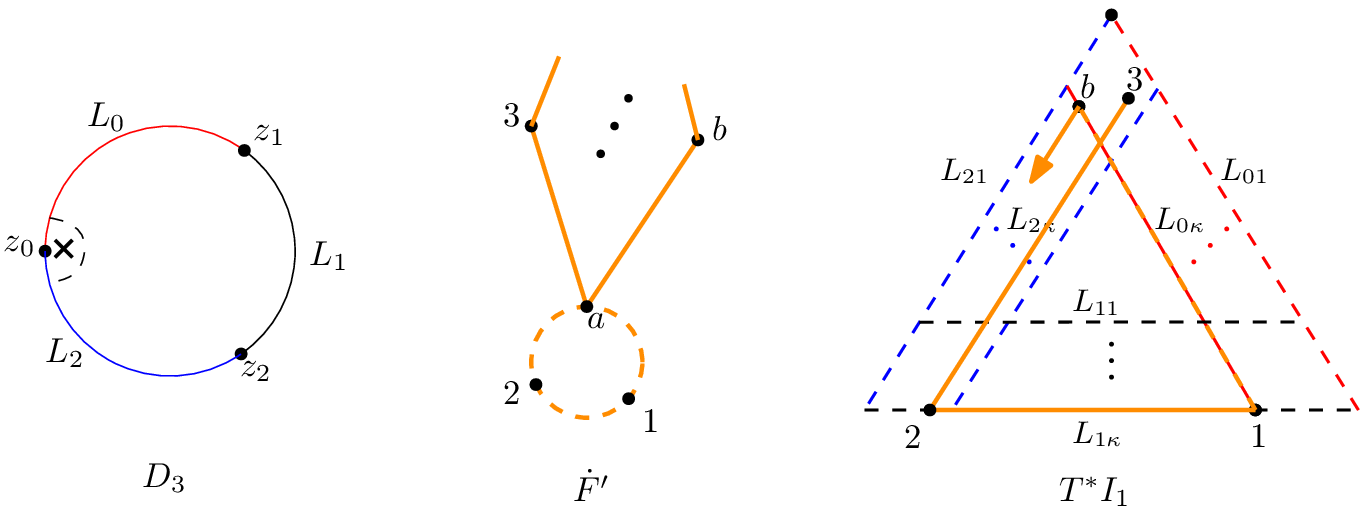}
        \caption{The case $t<\kappa-m$.}
        \label{fig-induction-1-3}
    \end{figure}
\end{enumerate}
This finishes the proof.
\end{proof}

The second case is when $w_1$ exchanges the last two elements.  
The schematic pictures are shown in Figures \ref{fig-induction-2} and \ref{fig-induction-22}, which correspond to two subcases depending on the action of $w_2$ on the last two elements.
The following proposition is a generalization of Lemmas \ref{lemma T1=T}, \ref{lemma TT}.

\begin{figure}
    \centering
    \includegraphics[width=10cm]{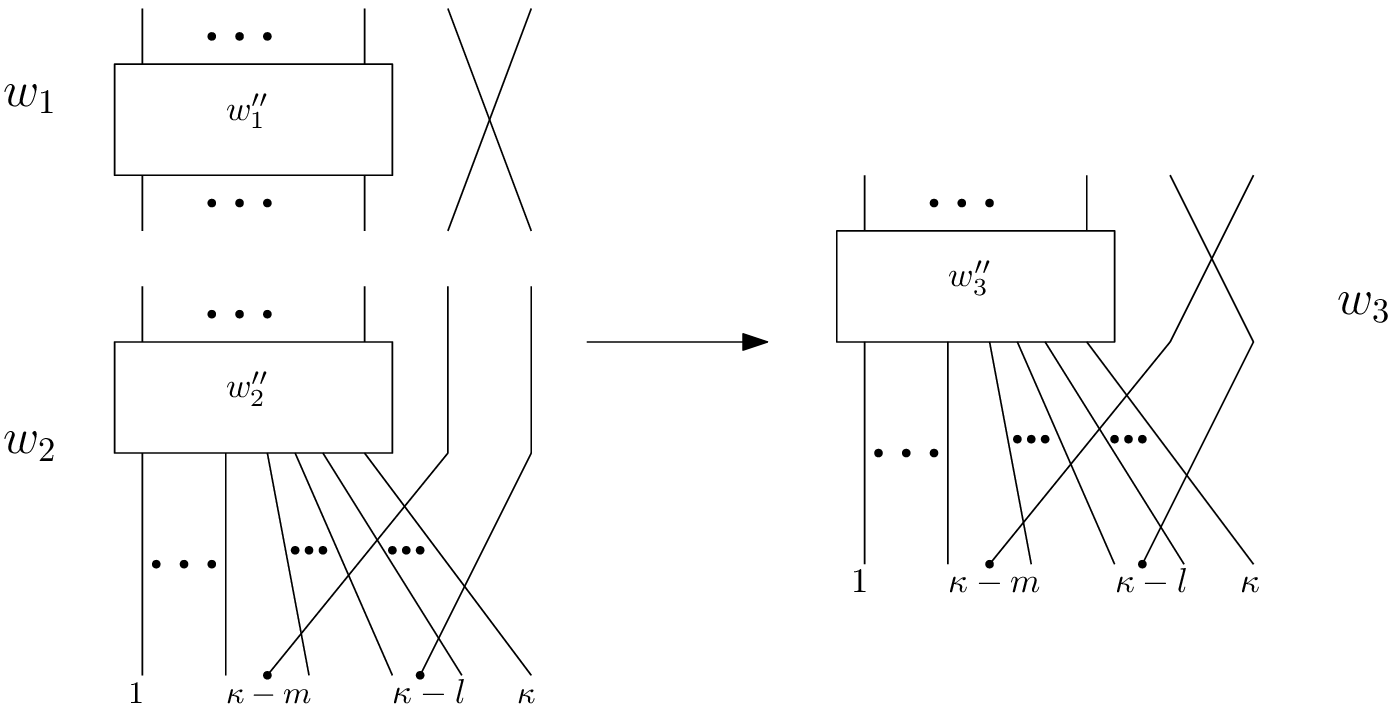}
    \caption{The case for Proposition \ref{prop-key2} (1).}
    \label{fig-induction-2}
\end{figure}    
\begin{figure}
    \centering
    \includegraphics[width=12cm]{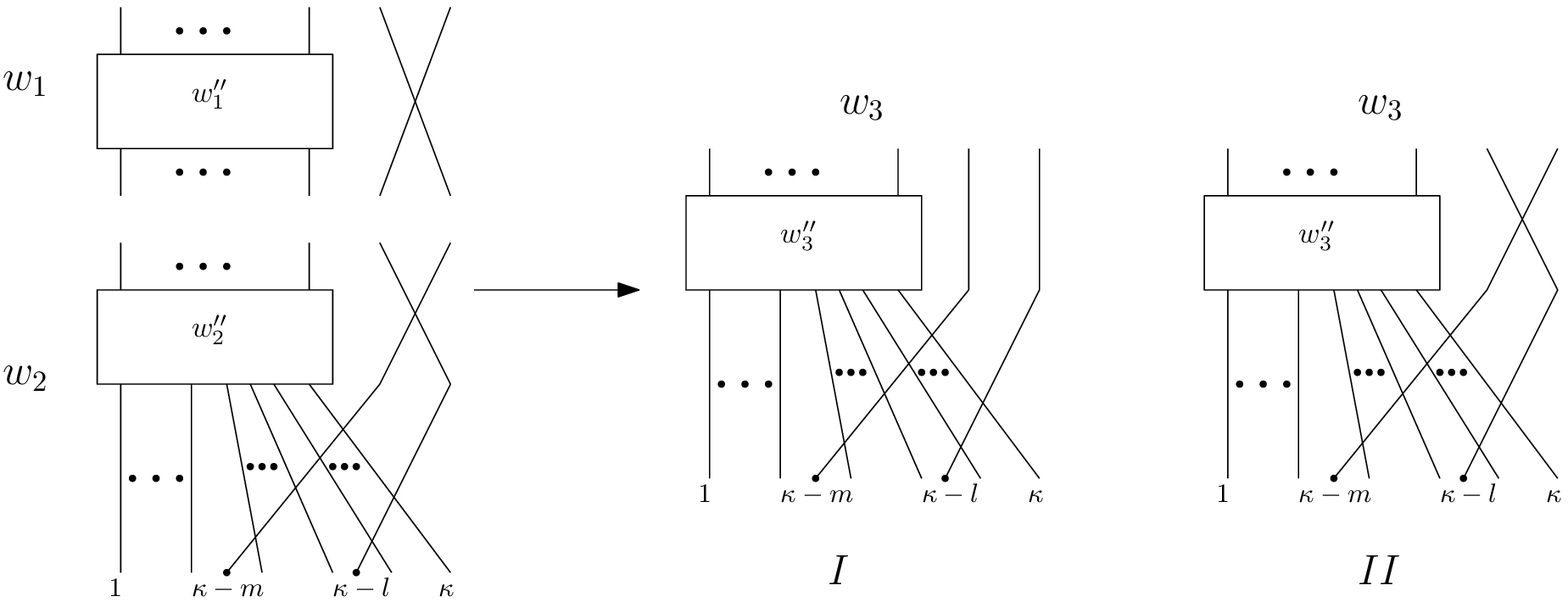}
    \caption{The case for Proposition \ref{prop-key2} (2).}
    \label{fig-induction-22}
\end{figure}

\begin{proposition} 
\label{prop-key2}
    For $w_1,w_2,w_3 \in S_{\kappa}$, suppose that $w_1=w_1''s_{\kappa-1}$, where $w_1'' \in S_{\kappa-2}$.
    \begin{enumerate}
        \item If $w_2=w_2''s_{\kappa-2}\cdots s_{\kappa-m}s_{\kappa-1}s_{\kappa-2}\cdots s_{\kappa-l}$, where $w_2'' \in S_{\kappa-2}$, $m > l \ge 0$,
        we have 
        \begin{align*} 
           & \#\mathcal{M}^{\chi}(T_{w_1}, T_{w_2}, T_{w_3}) \\
           =&\left\{
            \begin{array}{ll}
                \#\mathcal{M}^{\chi-2}(T_{w_1''}, T_{w_2''}, T_{w_3''}) & \mathrm{if}\,\,\, w_3=w_3''s_{\kappa-1}s_{\kappa-2}\cdots s_{\kappa-m}s_{\kappa-1}s_{\kappa-2}\cdots s_{\kappa-l},\\
                0 & \mathrm{otherwise},
           \end{array}\right.
        \end{align*}
        where $w_3'' \in S_{\kappa-2}$.
        \item If $w_2=w_2''s_{\kappa-1}s_{\kappa-2}\cdots s_{\kappa-m}s_{\kappa-1}s_{\kappa-2}\cdots s_{\kappa-l}$, where $w_2'' \in S_{\kappa-2}$, $m > l \ge 0$,
        we have 
        \begin{align*} 
            & \#\mathcal{M}^{\chi}(T_{w_1}, T_{w_2}, T_{w_3}) \\
            =&\left\{
            \begin{array}{ll}
                \#\mathcal{M}^{\chi-2}(T_{w_1''}, T_{w_2''}, T_{w_3''}) & \mathrm{if}\,\,\, w_3=w_3''s_{\kappa-2}\cdots s_{\kappa-m}s_{\kappa-1}s_{\kappa-2}\cdots s_{\kappa-l},\\
                \#\mathcal{M}^{\chi-1}(T_{w_1''}, T_{w_2''}, T_{w_3''}) & \mathrm{if}\,\,\, w_3=w_3''s_{\kappa-1}s_{\kappa-2}\cdots s_{\kappa-m}s_{\kappa-1}s_{\kappa-2}\cdots s_{\kappa-l},\\
               0 & \mathrm{otherwise},
           \end{array}\right.
       \end{align*}
        where $w_3'' \in S_{\kappa-2}$.
    \end{enumerate}               
\end{proposition}
\begin{proof}
    The proof is similar to but slightly longer than that of Proposition \ref{prop-key1} since we need to discuss the last two strands of $w_1$ instead of one. 
    We keep track of the following labels.
    \begin{itemize}
        \item The strands of $w_3$ which start from the positions $\kappa$ and $\kappa-1$ end on positions $t_1$ and $t_2$, respectively.
        \item The strands of $w_3$ which end on the position $\kappa-m$ and $\kappa-l$ start from positions $r_1$ and $r_2$, respectively.
    \end{itemize}
    
    Figure \ref{fig-induction-2-general} describes the part of generators $T_{w_1},T_{w_2}$ corresponding to the last two strands of $w_1$, where the dashed circles describe the undetermined $T_{w_3}$.
We discuss Cases (1) and (2) separately in the following. 

    \begin{figure}
        \centering
        \includegraphics[width=14cm]{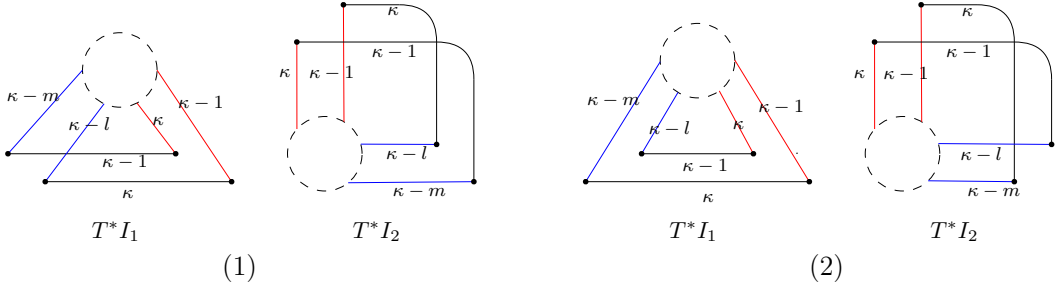}
        \caption{The part of generators $T_{w_1},T_{w_2}$ for Cases (1) and (2).}
        \label{fig-induction-2-general}
    \end{figure}
    
%

    \begin{figure}
        \centering
        \includegraphics[width=15cm]{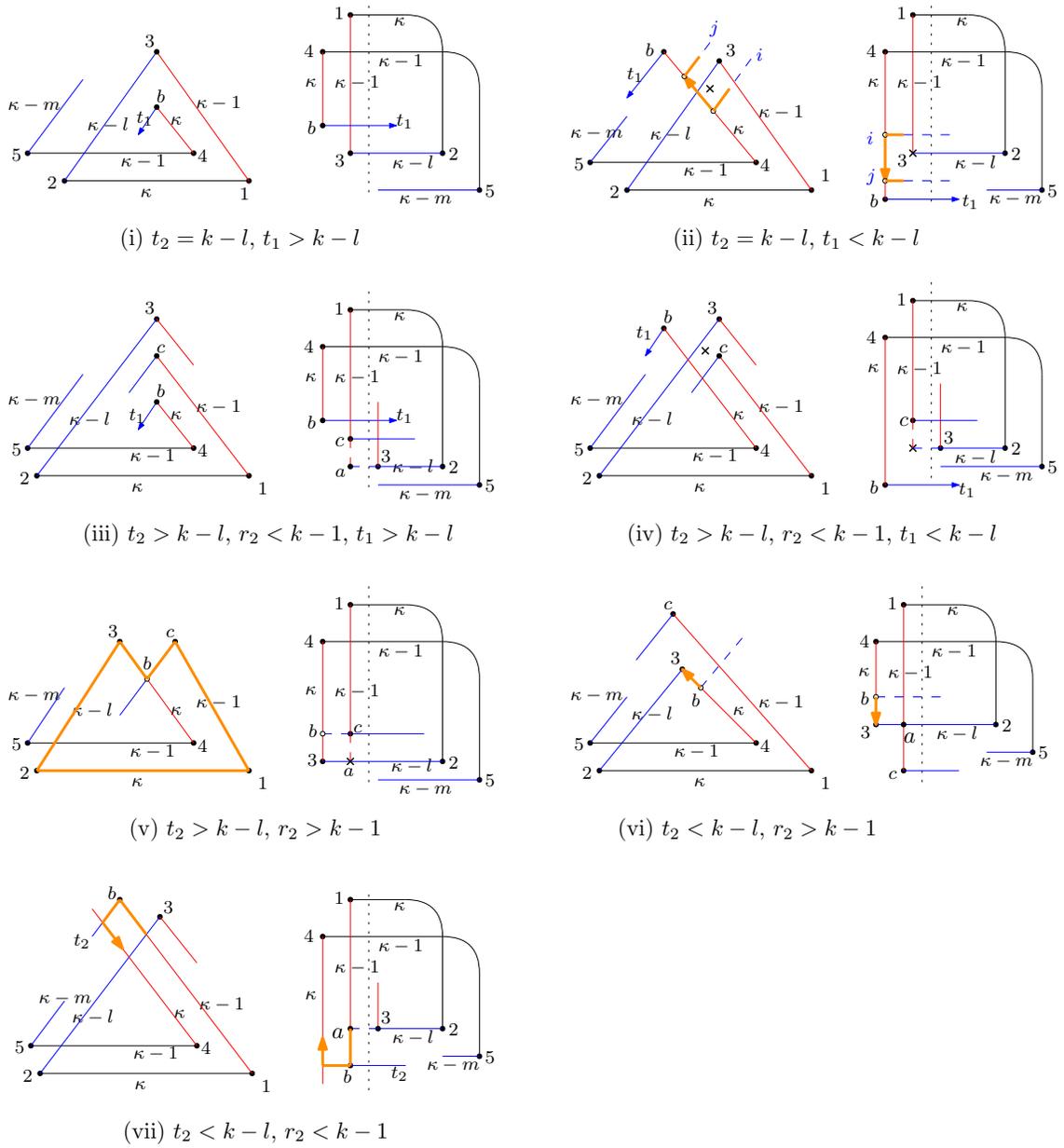}
        \caption{The subcases of Case (1).}
        \label{fig-induction-2-1}
    \end{figure}
    
    ~\\
    \noindent
    (1)  We first consider the case when $t_1=\kappa-m$, $t_2=\kappa-l$. 
This is equivalent to 
$ w_3=w_3''s_{\kappa-1}s_{\kappa-2}\cdots s_{\kappa-m}s_{\kappa-1}s_{\kappa-2}\cdots s_{\kappa-l},$
    for some $w_3''\in S_{\kappa-2}$.
    Consider a holomorphic curve in $\mathcal{M}^{\chi}(T_{w_1}, T_{w_2}, T_{w_3})$ which contains two trivial disks corresponding to the last two strands of $w_1$. 
    The remaining components represent a curve in $\mathcal{M}^{\chi-2}(T_{w_1''}, T_{w_2''}, T_{w_3''})$.
Thus, $\mathcal{M}^{\chi-2}(T_{w_1''}, T_{w_2''}, T_{w_3''})$ can be viewed as a subset of $\mathcal{M}^{\chi}(T_{w_1}, T_{w_2}, T_{w_3})$. 
We show that no other curve exists in the rest of the proof.
    The subcases are shown in Figure \ref{fig-induction-2-1}.

    \begin{enumerate}[label=(\roman*)]
        \item $t_2=k-l$, $t_1>k-l$.
        As $d\to0$, $\dot{F}_{(12)}$ bubbles off as a triangle with vertices $\{p_1,p_2,p_a\}$, where $p_a\in\dot{F}'$ is the nodal point mapped to the limit of $q_1'$ and $q_2'$ in the $T^*I_1$-direction.
        Then the projection of $\dot{F}'_{(3)}$ to the $T^*I_2$-direction must be the constant map to $q_3''$. 
        Since $\pi_{T^*I_2}\circ u$ is of degree 0 or 1 near $q_3$, the image $\pi_{T^*I_2}\circ u(\dot{F}\backslash(\dot{F}'_{(12)}\cup \dot{F}'_{(3)}))$ is disjoint from $q''_3$. 
        It implies that $\dot{F}'_{(12)}\cup \dot{F}'_{(3)}$ is a connected component of $\dot{F}'$. 
        Therefore, the triangle $\{p_1,p_2,p_3\}$ forms a connected component of $\dot{F}$ before the degeneration.
        By removing the triangle $\{p_1,p_2,p_3\}$, the problem reduces to the case (2) of Proposition \ref{prop-key1} with $\kappa-1$-strands. Hence, $\#\mathcal{M}^{\chi}(T_{w_1}, T_{w_2}, T_{w_3})=0$.
        \item $t_2=k-l$, $t_1<k-l$.
        As $d\to0$, $\dot{F}_{(12)}$ bubbles off as a triangle with vertices $\{p_1,p_2,p_a\}$ and the projection of $\dot{F}'_{(3)}$ to the $T^*I_2$-direction must be the constant map to $q''_3$.
        Moreover $\dot{F}'_{(3)}$ is a bigon with possible nodal degeneration points which are connected to other irreducible components of $\dot{F}'$.
        Denote one of such nodal points on $\dot{F}'$ by $p_n$, whose images in $T^*I_1$ and $T^*I_2$ are drawn as the crossings in Figure \ref{fig-induction-2-1} (ii).
        We now remove the bigon $\dot{F}'_{(3)}$ from $\dot{F}'$ but keep $p_n$ remained.
        We denote the irreducible component containing $p_n$ in the remaining part of $\dot{F}'$ by $\dot{F}'_{p_n}$. 
        
        In the $T^*I_2$-direction, the projection of $u(\dot{F}'\backslash\dot{F}'_{(3)})$ to the left side of the vertical dotted line is of degree 1.
        Let $C$ be the boundary of the image $\pi_{T^*I_2}\circ(\dot{F}'_{p_n})$. 
        Then the part of $C$ near $L_{0{\kappa}}\cap L_{2(\kappa-l)}$ is locally drawn as the orange lines, which goes from $L_{2i}$ to $L_{2j}$ on $L_{0\kappa}$ for $i>\kappa-l$, $j<\kappa-m$.
        We denote the preimage of the orange arrow from $L_{2i}$ to $L_{2j}$ by $C_{arrow}$. It has the positive boundary orientation.

 In the $T^*I_1$-direction, the position of the crossing must be above $L_{0\kappa}$ since $\pi_{D_3}\circ u(p_n)=z_0$.
        However, the image of $C_{arrow}$, denoted by the orange arrow, has the negative boundary orientation. This leads to a contradiction.
        Therefore, $\#\mathcal{M}^{\chi}(T_{w_1}, T_{w_2}, T_{w_3})=0$.      
        \item $t_2>k-l$, $r_2<k-1$, $t_1>k-l$.
        As $d\to0$, $\dot{F}_{(12)}$ bubbles off as a triangle with vertices $\{p_1,p_2,p_a\}$.
        Since on $T^*I_2$, $\pi_{T^*I_2}\circ u$ is of degree 0 near the intersection of the extension of $(q''_1q''_c)$ and $(q''_2q''_3)$, $\{p_1,p_2,p_a\}$ cannot form a triangle. This leads to a contradiction.
        Therefore, $\#\mathcal{M}^{\chi}(T_{w_1}, T_{w_2}, T_{w_3})=0$.
        \item $t_2>k-l$, $r_2<k-1$, $t_1<k-l$.
        As $d\to0$, $\dot{F}_{(12)}$ bubbles off as a triangle with vertices $\{p_1,p_2,p_a\}$, where $p_a$ is mapped to a point in $T^*I_2$, denoted by a crossing.
        We denote the preimage of this crossing in the irreducible component other than $\dot{F}_{(12)}$ and $\dot{F}_{(3)}$ by $p_n$.
        The image of $p_n$ in $T^*I_1$ is also denoted by a crossing. It sits above $L_{0\kappa}$ for the same reason as in (ii).
        The remaining argument is the same as in (ii). We conclude that $\#\mathcal{M}^{\chi}(T_{w_1}, T_{w_2}, T_{w_3})=0$.
        \item $t_2>k-l$, $r_2>k-1$.
        As $d\to0$, $\dot{F}_{(12)}$ bubbles off as a triangle $T$ with vertices $\{p_1,p_2,p_a\}$, where $p_a$ is mapped to the crossing in $T^*I_2$.
        The other irreducible component of $\dot{F}'$ containing $p_a$ is the quadrilateral $Q$ with vertices $\{p_3,p_c,p_a,p_b\}$ which is the bottom-left part in the $T^*I_2$-direction. 
        Figure \ref{fig-induction-2-1-v} describes the degenerated domain $\dot{F}'$.
        
        Removing $T$ and $Q$ from $\dot{F}'$ corresponds to removing the orange polygon in the $T^*I_1$-direction. 
        As a result, the vertices $\{p_1,p_2,p_3,p_c\}$ are replaced by $p_b$. Then the problem is reduced to the case (2) of Proposition \ref{prop-key1} with $\kappa-1$ strands.
        Hence, $\#\mathcal{M}^{\chi}(T_{w_1}, T_{w_2}, T_{w_3})=0$.   
        \begin{figure}
            \centering
            \includegraphics[width=6cm]{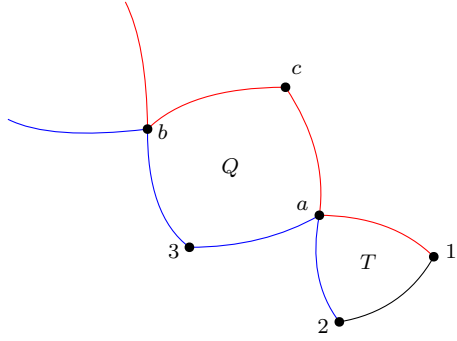}
        \caption{The subcase (v).}
            \label{fig-induction-2-1-v}
        \end{figure}
             \item $t_2<k-l$, $r_2>k-1$.
        This is similar to (ii). The orientation of the orange arrows leads to a contradiction. Hence $\#\mathcal{M}^{\chi}(T_{w_1}, T_{w_2}, T_{w_3})=0$.
        \item $t_2<k-l$, $r_2<k-1$.
        This is similar to (ii). So $\#\mathcal{M}^{\chi}(T_{w_1}, T_{w_2}, T_{w_3})=0$.
        \end{enumerate}
    
    \begin{figure}
        \centering
        \includegraphics[width=15cm]{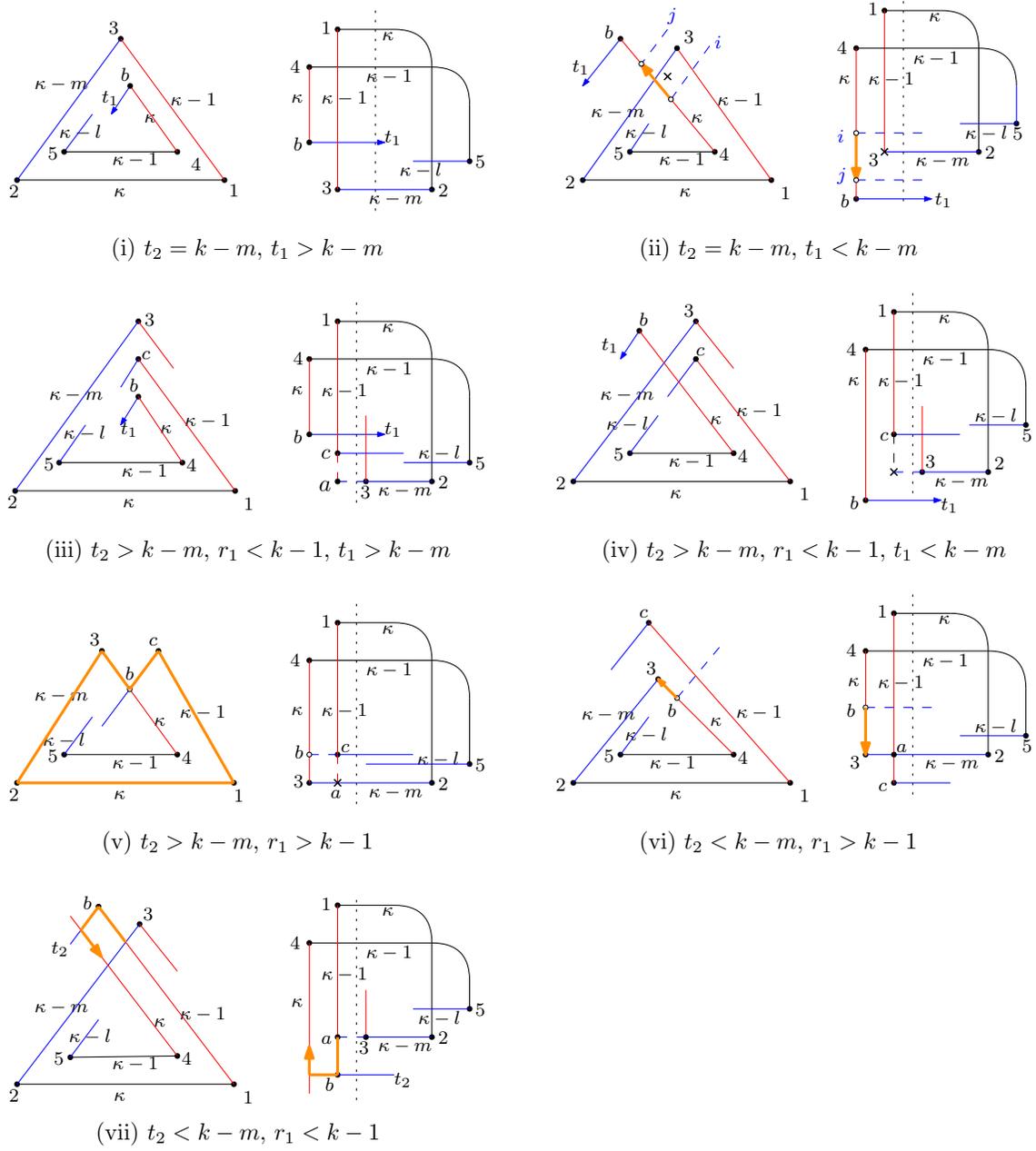}
        \caption{The subcases of Case (2).}
        \label{fig-induction-2-2}
    \end{figure}
    
    ~\\
    \noindent
    (2) The subcases are shown in Figure \ref{fig-induction-2-2}.
    The proofs of all subcases are similar to those in (1) except for the subcase (v). We discuss the subcase (v) only and omit the others.
    
    \begin{enumerate}[label=(\roman*)]
        \item[(v)] $t_2>k-m$, $r_1>k-1$.
        The proof is similar to that of the subcase (v) of (1).
         As $d\to0$, $\dot{F}_{(12)}$ bubbles off as a triangle $T$ with vertices $\{p_1,p_2,p_a\}$, where $p_a$ is mapped to the crossing in $T^*I_2$, and $\{p_3,p_c,p_a,p_b\}$ forms a quadrilateral $Q$, as the bottom-left part in the $T^*I_2$-direction. 
        Figure \ref{fig-induction-2-1-v} describes the degenerated domain $\dot{F}'$.
        Removing $T$ and $Q$ from $\dot{F}'$ corresponds to removing the orange polygon in the $T^*I_1$-direction. 
        As a result, the vertices $\{p_1,p_2,p_3,p_c\}$ are replaced by $p_b$. Then the problem is reduced to the case with $\kappa-1$ strands.
        There are three possibilities.
        \begin{enumerate}
            \item $t_2=\kappa-l$. This is similar to the case (1) of Proposition \ref{prop-key1}.
            If the limiting curve exists, then $\{p_1,p_2,p_3,p_4,p_5,p_c\}$ must forms a (hexagon) disk component $H$ of $\dot{F}$.
            The count of $u\in\mathcal{M}^{\chi}(T_{w_1}, T_{w_2}, T_{w_3})$ restricted to $H$ is exactly the count of $\mathcal{M}^{\chi=1}_{J}(T_1,T_1,T_1)$ in Lemma \ref{lemma TT}, which is equals 1.
            The count of $u$ restricted to $\dot{F}\backslash H$ is the count of $\mathcal{M}^{\chi-1}(T_{w_1''}, T_{w_2''}, T_{w_3''})$.
            Therefore, $\#\mathcal{M}^{\chi}(T_{w_1}, T_{w_2}, T_{w_3})=\#\mathcal{M}^{\chi-1}(T_{w_1''}, T_{w_2''}, T_{w_3''})$.
            \item $t_2>\kappa-l$. This is similar to the case (2) of Proposition \ref{prop-key1}. So $\#\mathcal{M}^{\chi}(T_{w_1}, T_{w_2}, T_{w_3})=0$.
            \item $t_2<\kappa-l$. This is similar to the case (3) of Proposition \ref{prop-key1}. So $\#\mathcal{M}^{\chi}(T_{w_1}, T_{w_2}, T_{w_3})=0$.
      \end{enumerate}
  \end{enumerate}
  This finishes the proof.
\end{proof}

The following corollaries are direct consequences by inductively using the two propositions above.

\begin{corollary} \label{cor1}
The generator $T_{{\operatorname{id}}}$ is the identity in $\operatorname{End}(L^{\otimes \kappa})$.
\end{corollary}

\begin{corollary} \label{cor2}
We have
       $ T_iT_w=\left\{
        \begin{array}{ll}
        T_{s_iw} & \mathrm{if}\,\,\,l(s_iw)>l(w)+1,\\
               T_{s_iw}+\hbar T_w & \mathrm{if}\,\,\,l(s_iw)<l(w)-1.
               \end{array}\right.$
\end{corollary}

\begin{corollary} \label{cor3}
The generators $T_i$ satisfy the relations in the Hecke algebra:
    \begin{align}
        T_{i}^2&=1+\hbar T_i,\label{eq-x-x}\\
        T_iT_j&=T_jT_i\,\,\,\mathrm{for}\,\,\,|i-j|>1,\label{eq-2-commute}\\
        T_iT_{i+1}T_i&=T_{i+1}T_{i}T_{i+1}\label{eq-3-commute}.
    \end{align}
\end{corollary}

\begin{proof}[Proof of Theorem \ref{main-thm}]
    Define a unital $\mathbb{Z}[\hbar]$-algebra map $\phi: H_{\kappa} \to \operatorname{End}(L^{\otimes \kappa})$ on the algebra generators by $\phi(\tilde{T}_i)=T_i$. 
    The map is well-defined by Corollary \ref{cor3}.
    The multiplication rules on $H_{\kappa}$ in (\ref{eq-hecke}) and that of $\operatorname{End}(L^{\otimes \kappa})$ in Corollary \ref{cor2} are the same. So we have $\phi(\tilde{T}_w)=T_w$ for all $w \in S_{\kappa}$.
\end{proof}

\printbibliography

\end{document}